\documentclass[a4paper,10pt]{article}
\usepackage[utf8]{inputenc}

\usepackage{cmap}
\usepackage[T1]{fontenc}

\usepackage{tikz}

\usepackage{subcaption}

\usepackage{tocloft}

\usepackage[english]{babel}
\usepackage{amsfonts}
\usepackage{amssymb, amsthm}
\usepackage{xcolor,graphicx}
\usepackage{marginnote}
\usepackage{amsmath}
\usepackage{a4wide}
\usepackage[affil-it]{authblk}

\usepackage{epstopdf}
\newcommand{\doi}[1]{\href{http://dx.doi.org/#1}{doi:#1}}

\usepackage{titlesec}
\titlelabel{\thetitle.\quad}

\usepackage[labelsep=period]{caption}

\usepackage{enumitem}

\usepackage{crossreftools}
\makeatletter
\newcommand{\optionaldesc}[2]{%
	\phantomsection
	#1\protected@edef\@currentlabel{#1}\label{#2}%
}
\makeatother

\numberwithin{equation}{section}

\let\OLDthebibliography\thebibliography
\renewcommand\thebibliography[1]{
	\OLDthebibliography{#1}
	\setlength{\parskip}{1pt}
	\setlength{\itemsep}{1pt plus 0.3ex}
}

\usepackage[colorlinks=true,linktocpage,pdfpagelabels,
bookmarksnumbered,bookmarksopen]{hyperref}
\definecolor{ForestGreen}{rgb}{0.1,0.6,0.05}
\definecolor{EgyptBlue}{rgb}{0.063,0.1,0.6}
\hypersetup{
	colorlinks=true,
	linkcolor=EgyptBlue,         
	citecolor=ForestGreen,
	urlcolor=olive
}

\usepackage[hyperpageref]{backref}

\allowdisplaybreaks
\usepackage{mathtools}
\mathtoolsset{showonlyrefs}

\usepackage{orcidlink}

\usepackage{accents}

\def\Wo{W_0^{1,p}(\Omega)}
\def\We{W_0^{1,p}}
\def\W{W^{1,p}(\mathbb{R}^N)}
\def\XO{\Wo}

\newtheorem{theorem}{Theorem}[section]
\newtheorem{lemma}[theorem]{Lemma}

\theoremstyle{definition}
\newtheorem{remark}[theorem]{Remark}

\title{\vspace*{-5ex}
On the nodal set conjecture for the $p$-Laplacian in circularly symmetric domains
}

\author{Vladimir Bobkov\\}
\date{}

\AtEndDocument{%
	\par
	\medskip
	\bigskip
	\begin{tabular}{@{}l@{}}%
		(V.~Bobkov)\\[0.2em]
		\textsc{Institute of Mathematics, Ufa Federal Research Centre, RAS}\\
		\textsc{Chernyshevsky str. 112, 450008 Ufa, Russia}\\[0.3em]
		\orcidlink{0000-0002-4425-0218} 0000-0002-4425-0218\\[0.3em]
		\textit{E-mail address}: \texttt{bobkov@matem.anrb.ru, bobkovve@gmail.com}\\[1.5em]
\end{tabular}}

\begin{document}
	\maketitle
	\vspace*{-5ex}
	\begin{abstract}
		
		In 1990, P\"utter shown that the nodal line of any second eigenfunction of the Dirichlet Laplacian on a planar bounded simply connected domain $\Omega$ intersects the boundary $\partial\Omega$ provided $\Omega$ has the circular symmetry. By adopting the method of moving polarization, we establish similar information on the nodal set of second eigenfunctions of the Dirichlet $p$-Laplacian on circularly symmetric domains in arbitrary higher dimension. 
		
		
	\end{abstract}

	\section{Introduction}\label{sec:intro}

	Let $\Omega$ be a bounded domain in $\mathbb{R}^N$, $N \geqslant 2$, and let $p \in (1,+\infty)$.
	Consider the eigenvalue problem
	\begin{equation}\label{eq:D}
		\tag{$\mathcal{D}$}
		\left\{
		\begin{aligned}
			-\Delta_p u &= \lambda |u|^{p-2} u 
			&&\text{in } \Omega, \\
			u &= 0  &&\text{on } \partial \Omega,
		\end{aligned}
		\right.
	\end{equation}
where $\Delta_p u$ is the $p$-Laplacian. 
This problem is understood in the weak sense: a parameter $\lambda \in \mathbb{R}$ and a function $u \in \Wo \setminus \{0\}$ are called \textit{eigenvalue} and \textit{eigenfunction} of \eqref{eq:D}, respectively, provided
\begin{equation}\label{eq:D:weak}
\int_\Omega |\nabla u|^{p-2} \langle \nabla u, \nabla \phi \rangle \,d\mathbf{x}
= 
\lambda \int_\Omega |u|^{p-2} u \phi \,d\mathbf{x}
\quad \text{for any}~ \phi \in \Wo.
\end{equation}
Hereinafter, we use a bold single-character notation $\mathbf{x}$, $\mathbf{y}$, etc., for points from $\mathbb{R}^N$. 
Employing the Lusternik-Shnirelman theory, it can be shown that \eqref{eq:D} admits an infinite sequence of eigenvalues 
\begin{equation}\label{eq:lambdas}
0< \lambda_1 < \lambda_2 \leq \dots \leq \lambda_k \to +\infty
\quad \text{as}~ k \to +\infty,
\end{equation}
where $\lambda_1$ and $\lambda_2$ are the least eigenvalues of \eqref{eq:D}, see, e.g., \cite{drabekrobinson}, and we refer to \cite{anane1987,CuestaFucik} for the second inequality. 
Any eigenfunction $u$ of \eqref{eq:D} belongs to $C^{1}(\Omega)$, see, e.g., \cite{Dib}. 
Moreover, if for some $\mathbf{x} \in \partial \Omega$ and $\varepsilon>0$, $\Omega \cap B_\varepsilon(\mathbf{x})$ is of class $C^{1,\delta}$ for some $\delta \in (0,1)$, then $u \in C^{1}(\Omega \cup (\partial \Omega \cap B_\varepsilon(\mathbf{x})))$, as it follows from \cite[Section~3]{lieberman}. 
Hereinafter, $B_\varepsilon(\mathbf{x})$ stands for the open $N$-ball of radius $\varepsilon$ centered at $\mathbf{x}$.

While general qualitative properties of the first eigenfunction are relatively well understood, information on the behavior of higher eigenfunctions is more obscure, except that any higher eigenfunction $u$ is sign-changing (see, e.g., \cite{kawohl-lind}), that is, $u^\pm \not\equiv 0$. 
In what follows, we always denote 
$$
u^+ = \max\{u,0\} 
\quad
\text{and} 
\quad
u^- = \min\{u,0\},
$$ 
so that $u = u^+ + u^-$. 
We are interested in qualitative properties of \textit{second} eigenfunctions and note the following convenient characterization of the second eigenvalue $\lambda_2$, see \cite[Proposition~4.2]{Bob1} (or \cite[Lemma~3.1]{brasco-franzina1}) and \cite[p.~1165]{BK1}.
\begin{lemma}\label{lem:characterization:lambda}
		Let
		\begin{equation}\label{eq:lambda2-char}
			\mu_2 
			:=
			\inf 
			\left\{
			\max
			\left\{
			\frac{\int_\Omega |\nabla v^+|^p \,d\mathbf{x}}{\int_\Omega |v^+|^p \, d\mathbf{x}},
			\frac{\int_\Omega |\nabla v^-|^p \,d\mathbf{x}}{\int_\Omega |v^-|^p \, d\mathbf{x}}
			\right\}:~
			v \in \XO, ~ v^\pm \not\equiv 0
			\right\}.
		\end{equation}
		Then $\lambda_2 = \mu_2$, and for any minimizer $v$ of $\mu_2$ there exist $\alpha,\beta>0$ such that $\alpha v^+ + \beta v^-$ is a second eigenfunction of \eqref{eq:D}. 
\end{lemma}
Any second eigenfunction $u$ being continuous and sign-changing, we define its \textit{nodal set} as
\begin{equation*}
	\mathcal{Z}(u) = \overline{\{x \in \Omega:~ u(x) = 0\}}.
\end{equation*}
The aim of the present work is to describe a class of domains $\Omega$ for which the nodal set of $u$ intersects the boundary, that is, 
\begin{equation}\label{eq:dist0}
	\mathrm{dist}(\mathcal{Z}(u), \partial \Omega) = 0.
\end{equation}

In the linear case $p=2$, the validity of \eqref{eq:dist0} was conjectured by Payne in 1967, see \cite[Conjecture~5]{payne1967}. 
In general, Payne's conjecture (also known as the nodal set conjecture) is not true: there do exist domains whose second eigenfunction $u$ satisfies $\mathrm{dist}(\mathcal{Z}(u), \partial \Omega) > 0$. 
The first such example in the planar case was constructed in \cite{HoffmannOstenhof}, see also \cite{DGSH,freitas-krej1} and a recent breakthrough work \cite{freitas-leylekian}. For a higher-dimensional counterexample, we refer to \cite{fournais}, as well as to \cite{kennedy,mukh-saha} for subsequent developments. 

The results of \cite{kennedy,mukh-saha} indicate that, in the higher-dimensional, a counterexample to Payne's conjecture can be found in any \textit{topological} class. In the planar case, in view of \cite{freitas-leylekian}, an example violating \eqref{eq:dist0} can be found in the class of multiply connected domains with any fixed number of holes, while the validity of \eqref{eq:dist0} for planar simply connected domains remains an open problem. 

In the present work, we are interested in \textit{geometric} conditions on $\Omega$ for which Payne's conjecture is true. 
The first result in this regard was established by Payne himself in \cite{payne1973}, where he covered planar domains which are symmetric with respect to a line and convex in the perpendicular direction. 
In other words, Payne considered Steiner symmetric domains, see Remark~\ref{rem:sym} below. 
A higher-dimensional extension of this result was established by Damascelli \cite{damascelliNodal}.
P\"utter \cite{putter} proved the validity of \eqref{eq:dist0} for planar domains with circular symmetry (see the definition below) and under a few additional technical assumptions. 
Apparently, essentially the same result was later reestablished by Yang \& Guo in \cite{YG}\footnote{It is interesting to note that neither Payne nor P\"utter and Yang \& Guo used the terms ``Steiner symmetry'' and ``circular symmetry'', respectively, although both notions already existed at the time of appearance of \cite{payne1973,putter,YG}.}. 
Both proofs originate from the idea of \cite{payne1973} and essentially based on the linearity of the problem. 

We also refer to \cite{allesandrini,melas}, where \eqref{eq:dist0} was proved for planar convex domains. 
Planar doubly connected domains with certain symmetries were considered in \cite{kiwan,sarswat}. 
In the case of arbitrary dimension, a class of long and thin convex domains was covered in \cite{jerison1}, 
thin tubular neighborhoods of curves and hypersurfaces were studied in \cite{FK,krej-tusek}, dumbbell type domains with narrow connectors were investigated in \cite{mukh-saha}.

In the nonlinear case $p>1$, the validity of \eqref{eq:dist0} was proved in \cite{BK1} for Steiner symmetric domains, and this result was further extended to nonlinear nonlocal settings in \cite{BK2x}.  
We are not aware of other general results on the nodal set conjecture for the nonlinear problem \eqref{eq:D}.

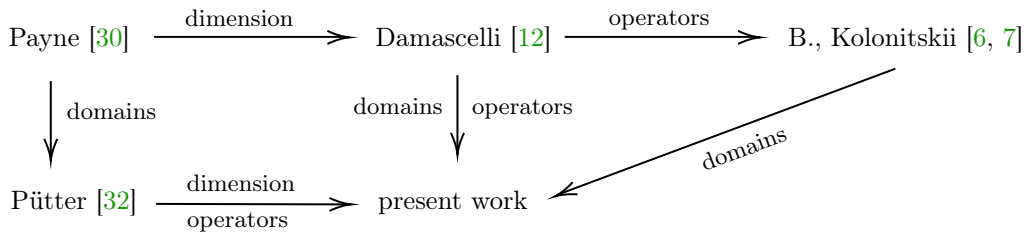
\begin{figure}[!ht]
\centering
\tikzset{every picture/.style={line width=0.75pt}} 

\begin{tikzpicture}[x=0.75pt,y=0.75pt,yscale=-1,xscale=1]

\draw    (130,41) -- (225.42,41) ;
\draw [shift={(227.42,41.25)}, rotate = 180.] [color={rgb, 255:red, 0; green, 0; blue, 0 }  ][line width=0.75]    (10.93,-3.29) .. controls (6.95,-1.4) and (3.31,-0.3) .. (0,0) .. controls (3.31,0.3) and (6.95,1.4) .. (10.93,3.29)   ;
\draw    (335,41) -- (430.42,41) ;
\draw [shift={(432.42,41.25)}, rotate = 180.] [color={rgb, 255:red, 0; green, 0; blue, 0 }  ][line width=0.75]    (10.93,-3.29) .. controls (6.95,-1.4) and (3.31,-0.3) .. (0,0) .. controls (3.31,0.3) and (6.95,1.4) .. (10.93,3.29)   ;
\draw    (78.42,63.25) -- (78.42,101.25) ;
\draw [shift={(78.42,103.25)}, rotate = 270] [color={rgb, 255:red, 0; green, 0; blue, 0 }  ][line width=0.75]    (10.93,-3.29) .. controls (6.95,-1.4) and (3.31,-0.3) .. (0,0) .. controls (3.31,0.3) and (6.95,1.4) .. (10.93,3.29)   ;
\draw    (281.42,60.25) -- (281.42,98.25) ;
\draw [shift={(281.42,100.25)}, rotate = 270] [color={rgb, 255:red, 0; green, 0; blue, 0 }  ][line width=0.75]    (10.93,-3.29) .. controls (6.95,-1.4) and (3.31,-0.3) .. (0,0) .. controls (3.31,0.3) and (6.95,1.4) .. (10.93,3.29)   ;
\draw    (131,125) -- (226.42,125.25) ;
\draw [shift={(228.42,125.25)}, rotate = 180] [color={rgb, 255:red, 0; green, 0; blue, 0 }  ][line width=0.75]    (10.93,-3.29) .. controls (6.95,-1.4) and (3.31,-0.3) .. (0,0) .. controls (3.31,0.3) and (6.95,1.4) .. (10.93,3.29)   ;
\draw    (500,57) -- (331,120.52) ;
\draw [shift={(330,121.25)}, rotate = 338] [color={rgb, 255:red, 0; green, 0; blue, 0 }  ][line width=0.75]    (10.93,-3.29) .. controls (6.95,-1.4) and (3.31,-0.3) .. (0,0) .. controls (3.31,0.3) and (6.95,1.4) .. (10.93,3.29)   ;

\draw (56,34) node [anchor=north west][inner sep=0.75pt]   [align=left] {Payne \cite{payne1973}};
\draw (239,34) node [anchor=north west][inner sep=0.75pt]   [align=left] {Damascelli \cite{damascelliNodal}};
\draw (445,34) node [anchor=north west][inner sep=0.75pt]   [align=left] {B., Kolonitskii \cite{BK1,BK2x}};
\draw (57,116) node [anchor=north west][inner sep=0.75pt]   [align=left] {P\"utter \cite{putter}};
\draw (240.4,117) node [anchor=north west][inner sep=0.75pt]   [align=left] {present work};
\draw (144,26) node [anchor=north west][inner sep=0.75pt]  [font=\small] [align=left] {dimension};
\draw (144.8,108.6) node [anchor=north west][inner sep=0.75pt]  [font=\small] [align=left] {dimension};
\draw (145,127.8) node [anchor=north west][inner sep=0.75pt]  [font=\small] [align=left] {operators};
\draw (85,72.6) node [anchor=north west][inner sep=0.75pt]  [font=\small] [align=left] {domains};
\draw (228,70.2) node [anchor=north west][inner sep=0.75pt]  [font=\small] [align=left] {domains};
\draw (287.2,71) node [anchor=north west][inner sep=0.75pt]  [font=\small] [align=left] {operators};
\draw (355.4,26) node [anchor=north west][inner sep=0.75pt]  [font=\small] [align=left] {operators};
\draw (400,100) node [anchor=north west][inner sep=0.75pt]  [font=\small,rotate=-339] [align=left] {domains};

\end{tikzpicture}
\caption{Extensions of \cite{payne1973} for domains with reflection symmetry.}
\label{fig:refs}
\end{figure}

In the present work, we extend the results of \cite{putter,YG} on the validity of \eqref{eq:dist0} to arbitrary dimension, weaker regularity assumptions on $\Omega$, and nonlinear settings, by generalizing the moving polarization approach from \cite{BK1}, see Figure~\ref{fig:refs} for a schematic diagram. 
The polarization (equivalently, two-point rearrangement) is a symmetrization method which, being applied to a function or set, produces a new function or set with ``better'' properties, see Section~\ref{section:aux}. It found a variety of applications in the qualitative theory, see, e.g., \cite{baernstain,BartschWethWillem,brocksol,dubinin}, and its significant feature is independence from the linearity of the problem.

\medskip
Let us introduce a few notations. 
We decompose $\mathbb{R}^N = \mathbb{R}^2 \times \mathbb{R}^{N-2}$ and occasionally write a point $\mathbf{x} 
\in \mathbb{R}^N$ as $(x,y;z)$ with $(x,y) \in \mathbb{R}^2$ and $z \in \mathbb{R}^{N-2}$. 
For given $r>0$, $a \in\mathbb{R}$, and $\xi \in \mathbb{R}^{N-2}$, consider the (two-dimensional) circle of radius $r$ lying in the $(x,y)$-plane in $\mathbb{R}^N$ and centered at $(a,0;\xi)$:
$$
S_r^{2}(a,0;\xi) = \{(x,y;\xi) \in \mathbb{R}^N:~ (x-a)^2+y^2 = r^2\}.
$$
Let us also consider the half-hyperplane 
$$
H_a^{h} = \{(x,y;z) \in \mathbb{R}^N:~ x>a,~ y=0\}.
$$
We say that $\Omega$ is \textit{circularly symmetric} with respect to $H_{a}^{h}$ if for any $r>0$ and $\xi \in \mathbb{R}^{N-2}$ the intersection $\Omega \cap S_r^{2}(a,0;\xi)$ is either empty, or the full circle $S_r^{2}(a,0;\xi)$, or an arc of $S_r^{2}(a,0;\xi)$ whose middle point is $(a+r,0;\xi)$, cf.\ Figure~\ref{fig0}.

\begin{figure}[!ht]

\centering
\includegraphics[width=0.6\linewidth]{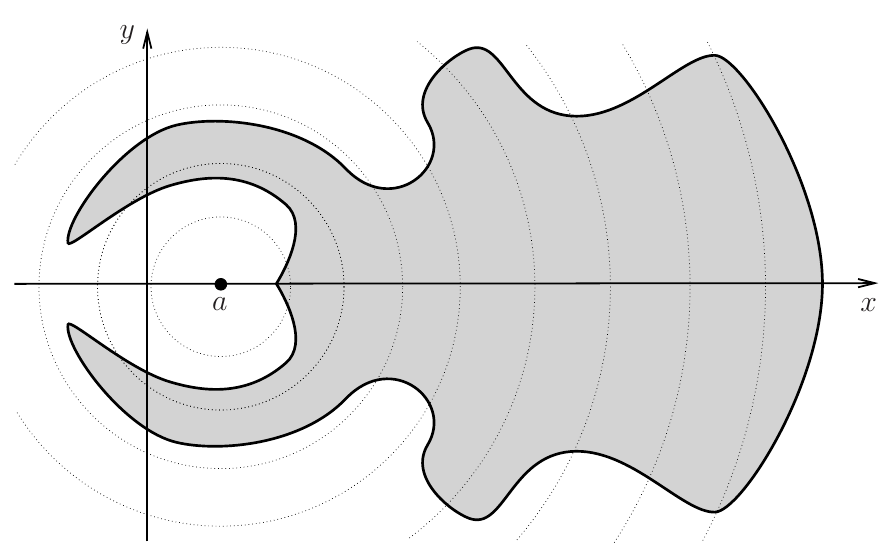}
\caption{An example of a domain with circular symmetry.}
\label{fig0}
\end{figure}

The circular symmetry (or, more generally, circular \textit{symmetrization}) was introduced by P\'olya \cite{Polya1}, and we refer to \cite{baernstain,sarvas} for developments, where the circular symmetrization is called the $(1,N)$- (or $1$-dimensional) \textit{cap symmetrization}. 
In the planar case, the circular symmetry simply means that the image of $\Omega$ in the polar coordinates $(\rho,\theta)$ is symmetric with respect to the line  $\{(\rho,\theta):\,\theta=0\}$ and convex in the $\theta$-direction (this characterization was used in \cite{YG}).
We also refer to \cite[Lemma~2.2]{putter} for another characterization of this symmetry class for sufficiently regular planar domains.  
In Lemma~\ref{lem:characterization:symmetry} below, we provide yet another useful characterization of circularly symmetric sets via the polarization. 

	\begin{remark}\label{rem:sym}
		The Steiner symmetry of $\Omega$ with respect to the hyperplane $\{(x,y;z):\,y=0\}$
		means that for any point $(x,0;z)$ the intersection of $\Omega$ with the line $\{(x,t;z):\, t \in \mathbb{R}\}$ is either empty or a segment centered at $(x,0;z)$. 
		In other words, such $\Omega$ is symmetric with respect to $\{(x,y;z):\, y=0\}$ and convex in the $y$-direction. 
		The notion of circular symmetry with respect to $H_{a}^h$ reduces to the Steiner symmetry with respect to $\{(x,y;z):\, y=0\}$ in the formal limiting case $a = -\infty$. 
		We also refer to \cite[Section~4.13]{sarvas} for another relations between these two notions. 
		In practice, the direct passage to the limit might not be convenient, and it is common to work with these tools separately, cf.\ \cite{baernstain,brocksol,dubinin,sarvas}.
	\end{remark}

Now we introduce several assumptions on $\Omega$ and state our main result. 
\begin{itemize}\addtolength{\itemindent}{1em}
	\item[$(O_1)$]
	There exists $a \in \mathbb{R}$ such that
	$\Omega$ is circularly symmetric with respect to $H_a^h$.
	\item[$(O_2)$]
	For any $r>0$ and $\xi \in \mathbb{R}^{N-2}$, the intersection $\Omega \cap S_r^{2}(a,0;\xi)$ is not the full circle. 
	\item[$(O_3)$] If $\Omega$ satisfies the interior ball condition at $\mathbf{x} \in \partial \Omega$, then $\Omega$ is of class $C^{1,\delta}$ in a neighborhood of $\mathbf{x}$ for some $\delta \in (0,1)$.
	\item[$(O_4)$] $\Omega$ satisfies the exterior ball condition at any $\mathbf{x} \in \partial \Omega$.
\end{itemize}

\begin{theorem}\label{thm:main}
	Let the assumptions $(O_1)$--$(O_4)$ hold. 
	Then any second eigenfunction $u$ of \eqref{eq:D} satisfies
		\begin{equation}
			\mathrm{dist}(\mathcal{Z}(u), \partial \Omega) = 0.
		\end{equation}
\end{theorem}

We believe that the assumptions $(O_3)$ and $(O_4)$ are of technical nature and can be weakened up to some extent, cf.\ \cite{damascelliNodal}. 
At the same time, the assumption $(O_2)$ is vital, as indicated by the construction from \cite{freitas-leylekian}.

Theorem~\ref{thm:main} will be proved in Section~\ref{section:proof}, and we discuss main auxiliary tools in Section~\ref{section:aux}.

\section{Polarization}\label{section:aux}
For a fixed $a \in \mathbb{R}$, consider the family of affine hyperplanes
	$$
	H_{\theta} = \{(x,y;z) \in \mathbb{R}^N:~ (x-a) \sin \theta - y \cos \theta = 0\},
	$$
	where $\theta \in [-\pi,\pi]$.
Denote the open half-spaces separated by $H_\theta$ as 
$$
\Sigma_{\theta}^\pm = \{(x,y;z)\in \mathbb{R}^N:~ \pm ((x-a) \sin \theta - y \cos \theta) > 0\},
$$
see Figure~\ref{fig2}. 
We also denote by $\sigma_{\theta}(\mathbf{x})$ the reflection of a point $\mathbf{x} \in \mathbb{R}^N$ with respect to $H_{\theta}$:
\begin{equation}\label{eq:sigma}
\sigma_\theta(x,y;z)
=
((x-a) \cos(2\theta)+y\sin(2\theta)+a,(x-a)\sin(2\theta)-y\cos(2\theta);z).
\end{equation}

Let $v: \mathbb{R} \to \mathbb{R}$ be some measurable function. 
The polarization of $v$ with respect to $H_{\theta}$ is a function $P_{\theta}  v: \mathbb{R} \to \mathbb{R}$ defined as (see, e.g., \cite{baernstain,BartschWethWillem,brocksol})
	\begin{equation}\label{eq:Pol}
		P_{\theta}  v(\mathbf{x}) = 
		\begin{cases}
			\min \{v(\mathbf{x}), v(\sigma_{\theta}(\mathbf{x}))\}, &\mathbf{x} \in \Sigma_{\theta}^+,\\
			v(\mathbf{x}), &\mathbf{x} \in H_{\theta},\\
			\max \{v(\mathbf{x}), v(\sigma_{\theta}(\mathbf{x}))\}, &\mathbf{x} \in \Sigma_{\theta}^-.
		\end{cases}
	\end{equation}
For further convenience, we also consider the ``complementary'' polarization $\widetilde{P}_\theta v:=-P_{\theta} (-v)$ or, equivalently, $\widetilde{P}_\theta v := P_{\theta+\pi} v$. 
In other words, in the definition of $\widetilde{P}_\theta v$, the placement of $\min$ and $\max$ is reversed compared to \eqref{eq:Pol}.

On can define the polarization of a set as the polarization of the corresponding characteristic function.
Effectively, for a measurable set $\Omega$, $P_{\theta}  \Omega$ and $\widetilde{P}_\theta \Omega$ are defined as
\begin{equation}\label{eq:POm}
	P_{\theta}  \Omega = 
	\begin{cases}
		\Omega \cap \sigma_{\theta}(\Omega) &\text{in } \Sigma_{\theta}^+,\\
		\Omega &\text{on } H_{\theta},\\
		\Omega \cup \sigma_{\theta}(\Omega) &\text{in } \Sigma_{\theta}^-,
	\end{cases}
	\quad 
	\text{and}
	\quad
	\widetilde{P}_\theta \Omega = 
	\begin{cases}
		\Omega \cup \sigma_{\theta}(\Omega) &\text{in } \Sigma_{\theta}^+,\\
		\Omega &\text{on } H_{\theta},\\
		\Omega \cap \sigma_{\theta}(\Omega) &\text{in } \Sigma_{\theta}^-,
	\end{cases}
\end{equation}
respectively, see Figure~\ref{fig2}.
Note that $\widetilde{P}_\theta \Omega = \mathbb{R}^N \setminus (P_{\theta}  (\mathbb{R}^N \setminus \Omega))$ or, equivalently, $\widetilde{P}_\theta \Omega = P_{\theta+\pi} \Omega$.
It is not hard to see from \eqref{eq:POm} that $P_\theta$ and $\widetilde{P}_\theta$ respect the set inclusion.
\begin{lemma}\label{lemma_domain_monotonicity}
	If $\Omega_1 \subset \Omega_2$, then $P_\theta \Omega_1 \subset P_\theta \Omega_2$ and $\widetilde{P}_\theta \Omega_1 \subset \widetilde{P}_\theta \Omega_2$ for any $\theta$.
\end{lemma}

\begin{figure}[!ht]
\centering
\includegraphics[width=0.6\linewidth]{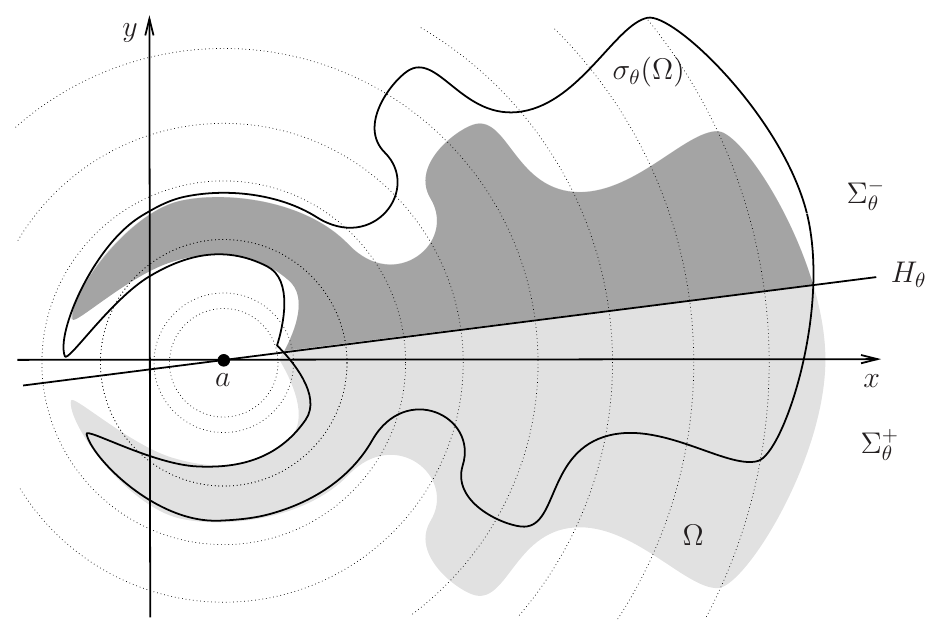}
\caption{Dark gray - $\Omega \cap \sigma_{\theta}(\Omega)$ in $\Sigma_{\theta}^-$, light gray - $\Omega \cup \sigma_{\theta}(\Omega)$ in $\Sigma_{\theta}^+$, so that $\widetilde{P}_\theta \Omega = \Omega$.}
\label{fig2}
\end{figure}

The following result provides a characterization of the circular symmetry via the polarization, see \cite[Lemma~2.2]{BK1} for an analogous characterization of the Steiner symmetry.
\begin{lemma}\label{lem:characterization:symmetry}
$\Omega$ is circularly symmetric with respect to $H_{a}^h$ if and only if 
either of the following three equivalent assumptions holds:
\begin{enumerate}[label={\rm(\roman*)}]
\item\label{lem:characterization:symmetry:1} 
$P_{\theta} \Omega = \Omega$ for all $\theta \in [-\pi/2,0]$ and $\widetilde{P}_\theta \Omega = \Omega$ for all $\theta \in [0,\pi/2]$.
\item\label{lem:characterization:symmetry:2} 
$P_{\theta} \Omega = \Omega$ for all $\theta \in [-\pi,0]$.
\item\label{lem:characterization:symmetry:3} 
$\widetilde{P}_\theta \Omega = \Omega$ for all $\theta \in [0,\pi]$.
\end{enumerate}
\end{lemma}
\begin{proof}
The assumptions \ref{lem:characterization:symmetry:1}, \ref{lem:characterization:symmetry:2}, \ref{lem:characterization:symmetry:3} are equivalent in view of the equality ${P}_\theta \Omega = \widetilde{P}_{\theta+\pi} \Omega$. Thus, it is sufficient to deal with only one of them. 

Let $\Omega$ be circularly symmetric with respect to $H_{a}^h$. Since for any $r>0$ and $\xi \in \mathbb{R}^{N-2}$ 
the intersection $\Omega \cap S_r^{2}(a,0;\xi)$ is either empty, or the full circle, or an arc whose middle point is $(a+r,0;\xi)$, we deduce from \eqref{eq:POm} that any such intersection does not change under the action of $P_{\theta}$ for $\theta \in [-\pi/2,0]$ and $\widetilde{P}_\theta$ for $\theta \in [0,\pi/2]$.  This gives the necessary part of the lemma.
	
Assume now that $\widetilde{P}_\theta \Omega = \Omega$ for all $\theta \in [0,\pi]$. 
Taking $\theta=0$ and $\theta=\pi$, we see from \eqref{eq:POm} that $\Omega$ is symmetric with respect to $H_{0}$. 
Suppose that there exist $r>0$ and $\xi \in \mathbb{R}^{N-2}$ such that the intersection $A := \Omega \cap S_r^{2}(a,0;\xi)$ has at least two connected components. 
Then one can find $\mathbf{x} \in S_r^{2}(a,0;\xi) \setminus  A$ and $\mathbf{y} \in A$ such that, in the polar coordinates, $\mathbf{x} = (\rho,\theta_1;\xi)$ and $\mathbf{y} = (\rho,\theta_2;\xi)$ with $0 \leq \theta_1 < \theta_2 \leq \pi$. 
Denoting $\theta_0 = (\theta_1+\theta_2)/2$, we see that $\sigma_{\theta_0} (\mathbf{x}) = \mathbf{y}$, which contradicts the assumption $\widetilde{P}_{\theta_0} \Omega = \Omega$. 
Therefore, for any $r>0$ and $\xi \in \mathbb{R}^{N-2}$, $\Omega \cap S_r^{2}(a,0;\xi)$ is either empty or a connected set (and hence either the full circle or an arc symmetric with respect to $H_0$). 
Finally, since $\widetilde{P}_{\pi/2} \Omega = \Omega$, we conclude that if, for some $r>0$ and $\xi \in \mathbb{R}^{N-2}$, $\Omega \cap S_r^{2}(a,0;\xi)$ is an arc, then its middle point is $(a+r,0;\xi)$. 
\end{proof}

Let us discuss a few properties of the polarization of functions. 
It is known that $(P_{\theta} v)^\pm = P_{\theta} (v^\pm)$, see, e.g., \cite[Lemma~2.1]{BartschWethWillem}, and hence we can write $P_{\theta}  v^\pm$, for short.
As a consequence, 
	\begin{equation}\label{eq:espantion_of_polarization_positive_negative} 
		P_{\theta}  (v^+ + v^-)
		= 
		P_{\theta}  v
		=
		(P_{\theta}  v)^+ + (P_{\theta}  v)^-
		= 
		P_{\theta}  v^+ + P_{\theta}  v^-.
	\end{equation}
	For any $v \in L^p(\mathbb{R}^N)$, we have 
	\begin{equation}\label{eq:weak-pol}
		\int_{\mathbb{R}^N} |P_{\theta} v^\pm|^p \,d\mathbf{x} 
		= 
		\int_{\mathbb{R}^N}  |v^\pm|^p \,d\mathbf{x},
	\end{equation}
and hence $P_{\theta} v \in L^p(\mathbb{R}^N)$, see \cite[Lemma~2.2]{BartschWethWillem} or \cite[Eq.~(3.7)]{brocksol}. 
	If $v \in \W$, then $P_{\theta} v \in \W$, and moreover
	\begin{equation}\label{eq:weak-po2}
\int_{\mathbb{R}^N} |\nabla P_{\theta} v^\pm|^p \,d\mathbf{x} = \int_{\mathbb{R}^N} |\nabla v^\pm|^p \,d\mathbf{x}, 
	\end{equation}
see \cite[Lemma~5.3]{brocksol} and \cite[Lemmas~2.2 and 2.3]{BartschWethWillem}. 
It is also possible to polarize functions from $\Wo$, where $\Omega$ is an open set, by considering their trivial extensions to $\mathbb{R}^N$.
Resulting functions belong to $\We(P_{\theta} \Omega \cup \widetilde{P}_\theta \Omega)$, as it follows from \cite[Corollary~5.1]{brocksol} and \eqref{eq:espantion_of_polarization_positive_negative}.
Using \cite[Lemma~5.6]{CuestaFucik}, this result can be clarified in the following way. 
\begin{lemma}\label{rem:lem:cuesta}
Let $v \in \Wo \cap C(\Omega)$ and let $\Omega^+_v := \{\mathbf{x} \in \Omega: v(\mathbf{x}) > 0\}$ and $\Omega^-_v := \{\mathbf{x} \in \Omega: v(\mathbf{x}) < 0\}$. Then $P_\theta v \in \We(P_\theta \Omega^+_v \cup \widetilde{P}_\theta \Omega^-_v)$ for any $\theta$.  
\end{lemma}

\section{Proof of Theorem~\ref{thm:main}}\label{section:proof}

Throughout this section, $u \in \Wo$ stands for a second eigenfunction of \eqref{eq:D}. 
Denote the nodal domains of $u$ as
$$
\Omega^+_u = \{\mathbf{x} \in \Omega:~ u(\mathbf{x})>0\}
\quad \text{and} \quad
\Omega^-_u = \{\mathbf{x} \in \Omega:~ u(\mathbf{x})<0\}.
$$
Suppose, by contradiction, that the nodal set of $u$ does not intersect the boundary of $\Omega$, that is,
\begin{equation}\label{eq:dist1}
\mathrm{dist}(\mathcal{Z}(u), \partial \Omega) > 0.
\end{equation}
Without loss of generality, we assume that $u < 0$ in near $\partial \Omega$, and hence $\overline{\Omega^+_u} \subset \Omega$. 
Our aim is to derive a contradiction, and the arguments are split in verification of the following two steps:

\textbf{Step I}.
Considering the polarization $P_{\theta} u$ of $u$, we increase $\theta$ until a certain value $\theta^*$ at which $\mathcal{Z}(P_{\theta^*} u)$ touches $\partial \Omega$ for the first time. 
From Lemma~\ref{lem:characterization:lambda}, we deduce that $\alpha P_{\theta^*} u^+ + \beta P_{\theta^*} u^-$ is a second eigenfunction of \eqref{eq:D} for some $\alpha,\beta>0$. 
In this step, we assume that $\Omega$ satisfies only the assumptions $(O_1)$ and $(O_2)$.

\textbf{Step II}. 
Additionally imposing the assumptions $(O_3)$ and $(O_4)$, we show that the function $\alpha P_{\theta_0} u^+ + \beta P_{\theta_0} u^-$ cannot be an eigenfunction of \eqref{eq:D} via arguments based on the boundary point lemma.

\medskip
Let us discuss \textbf{Step I}. 
For any point $\mathbf{x} \in \Omega$, consider
$$
\theta_{\mathbf{x}}
:= 
\inf
\left\{\theta > 0:~ 
\sigma_{\theta}(\mathbf{x})
\in \partial \Omega \right\}.
$$
Since $\Omega$ is symmetric with respect to $H_0$, we see that $\sigma_{0} (\mathbf{x}) \in \Omega$. 
Therefore, in view of the openness of $\Omega$ and the continuity of the mapping $\theta \mapsto \sigma_{\theta}(\mathbf{x})$, we get $\theta_{\mathbf{x}}>0$ and the infimum is attained. 
Thanks to the assumption $(O_2)$, we further obtain $\theta_{\mathbf{x}} < \pi$. 
(This is the only place where $(O_2)$ is used.)
Let us observe that the mapping $\mathbf{x} \mapsto \theta_{\mathbf{x}}$ is lower semicontinuous in $\Omega$.
Indeed, take any $\mathbf{x} \in \Omega$ and any sequence $\{\mathbf{x}_n\} \subset \Omega$ converging to $\mathbf{x}$. 
Since $\{\theta_{\mathbf{x}_n}\}$ is bounded, $\{\sigma_{\theta_{\mathbf{x}_n}}(\mathbf{x}_n)\} \subset \partial\Omega$, and $\partial\Omega$ is a compact set, we have, up to a subsequence,  
$$
\sigma_{\theta_{\mathbf{x}_n}}(\mathbf{x}_n) 
\to 
\sigma_{\lim \theta_{\mathbf{x}_n}}(\lim \mathbf{x}_n) \in \partial \Omega,
$$
and hence $\liminf \theta_{\mathbf{x}_n} \geq \theta_{\mathbf{x}}$. Since the initial choice of the sequence $\{\mathbf{x}_n\}$ is arbitrary, we obtain the desired lower semicontinuity.

Now we define\footnote{Since the unique continuation property for eigenfunctions of the $p$-Laplacian with $p \neq 2$ is unknown, we cannot directly say that $\mathcal{Z}(u) \subset \overline{\Omega^+_u}$, and this is the reason to use $\overline{\Omega^+_u} \cup \mathcal{Z}(u)$ in the definition of $\theta^*$.}
\begin{equation}\label{eq:theta-star}
\theta^*
:= 
\inf
\left\{\theta_{\mathbf{x}}:~ \mathbf{x} \in \overline{\Omega^+_u} \cup \mathcal{Z}(u) \right\},
\end{equation}
so that $\theta^* < \pi$. 
Due to our assumptions \eqref{eq:dist1} and $\overline{\Omega^+_u} \subset \Omega$, we also have $\theta^*>0$. 
Since $\overline{\Omega^+_u} \cup \mathcal{Z}(u)$ is a compact set and $\mathbf{x} \mapsto \theta_{\mathbf{x}}$ is lower semicontinuous, the infimum in $\theta^*$ is attained. 
Moreover, since $\Omega^+_u$ is open, the infimum is delivered by points from $\mathcal{Z}(u)$.
Denote the set of minimizers as $\mathcal{Y}(u)$, that is, 
\begin{equation}\label{def:Y(u)}
\mathcal{Y}(u) 
:= 
\left\{\mathbf{x} \in \mathcal{Z}(u):~ 
\sigma_{\theta^*} (\mathbf{x}) \in \partial \Omega \right\}.
\end{equation}
We do not know whether $\mathcal{Y}(u)$ is a singleton, and hence we define its relative boundary as 
$$
\partial \mathcal{Y}(u) 
:= 
\{\mathbf{x} \in \mathcal{Y}(u): \text{ any sufficiently small neighborhood of } \mathbf{x} \text{ contains } \mathbf{y} \in \mathcal{Z}(u) \setminus \mathcal{Y}(u)\},
$$
see Figure~\ref{fig1}.
Since $\mathcal{Z}(u)$ does not have isolated points (which follows from the strong maximum principle \cite[Theorem~5]{vazquez}) and \eqref{eq:dist1} holds, we see that $\partial \mathcal{Y}(u)$ is nonempty. 

\begin{lemma}\label{lemma_moving_polarization}
	Let $(O_1)$ and $(O_2)$ be satisfied. 
	Then $P_{\theta^*} u \in \Wo \setminus \{0\}$ 
	and $\mathrm{dist}(\mathcal{Z}(P_{\theta^*} u), \partial \Omega) = 0$. 
	Moreover, there exist $\alpha,\beta>0$ such that 
	$\alpha P_{\theta^*} u^+ + \beta P_{\theta^*} u^-$ is a second eigenfunction of \eqref{eq:D}.
\end{lemma}
\begin{proof}
By the definition of $\theta^*$, we have $\sigma_{\theta^*}(\overline{\Omega^+_u}) \subset \overline{\Omega}$, which implies that $\sigma_{\theta^*} (\Omega^+_u) \subset \Omega$. 
Therefore, noting that $\Omega^+_u \subset \Omega$, we deduce from \eqref{eq:POm} that 
\begin{equation}\label{eq:p+om}
{P}_{\theta^*} \Omega^+_u \subset \Omega.
\end{equation}
On the other hand, since $\theta^* \in (0,\pi)$, $\Omega^-_u \subset \Omega$, and $\Omega$ satisfies $(O_1)$, we get from Lemmas~\ref{lemma_domain_monotonicity} and \ref{lem:characterization:symmetry} that 
\begin{equation}\label{eq:p-om}
\widetilde{P}_{\theta^*}\Omega^-_u \subset \widetilde{P}_{\theta^*} \Omega = \Omega.
\end{equation}
Combining \eqref{eq:p+om} and \eqref{eq:p-om}, and recalling that $u \in C(\Omega)$, we obtain from Lemma~\ref{rem:lem:cuesta} that 
$$
P_{\theta^*} u \in \We({P}_{\theta^*} \Omega^+_u \cup \widetilde{P}_{\theta^*}\Omega^-_u) \subset \Wo,
$$
and we have $P_{\theta^*} u \not\equiv 0$ by construction. 
Using \eqref{eq:weak-pol}, \eqref{eq:weak-po2}, and recalling that $u$ satisfies \eqref{eq:D} with $\lambda=\lambda_2$, we also get
$$
\int_{\Omega} |\nabla P_{\theta} u^\pm|^p \,d\mathbf{x} 
= 
\int_{\Omega} |\nabla u^\pm|^p \,d\mathbf{x}
=
\lambda_2 \int_{\Omega}  |u^\pm|^p \,d\mathbf{x}
=
\lambda_2 \int_{\Omega} |P_{\theta} u^\pm|^p \,d\mathbf{x}.
$$
Consequently, $P_{\theta} u$ is a minimizer of $\mu_2$ defined in Lemma~\ref{lem:characterization:lambda}.
This lemma further yields the existence of $\alpha,\beta>0$ such that $\alpha P_{\theta^*} u^+ + \beta P_{\theta^*} u^-$ is another second eigenfunction of \eqref{eq:D}.
\end{proof}

\begin{figure}[!ht]
\centering
\includegraphics[width=0.6\linewidth]{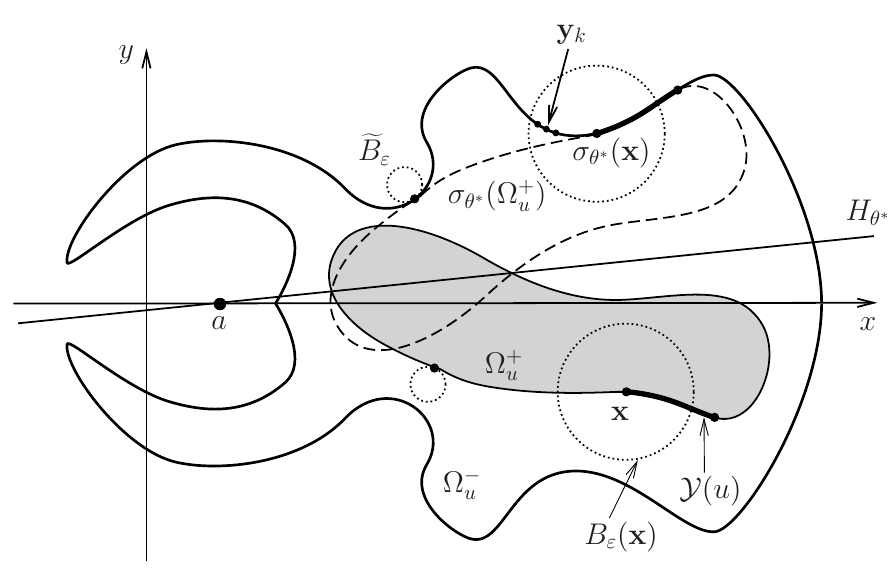}
\caption{Notations in the proof.}
\label{fig1}
\end{figure}

Let us now consider \textbf{Step II}.
We start with the following auxiliary result.
 \begin{lemma}\label{lemma_contact_set_properties}
	Let $(O_1)$, $(O_2)$, and $(O_4)$ be satisfied, and let $\mathbf{x} \in \mathcal{Y}(u)$. Then $|\nabla u(\mathbf{x})| > 0$ and hence $\mathbf{x} \in \partial \Omega^+_u \cap \partial \Omega^-_u$.
	Moreover, $\mathcal{Z}(u)$ is of class $C^2$ in a neighborhood of $\mathbf{x}$ and $\Omega^+_u$, $\Omega^-_u$ satisfy the interior ball condition at $\mathbf{x}$.
\end{lemma}
\begin{proof}
Let $\mathbf{x} \in \mathcal{Y}(u)$. 
In view of the exterior ball condition $(O_4)$, we can find a ball $\widetilde{B}_\varepsilon \subset \mathbb{R}^N \setminus \overline{\Omega}$ of radius $\varepsilon>0$ such that $\sigma_{\theta^*} (\mathbf{x}) \in \partial \widetilde{B}_\varepsilon \cap \partial \Omega$, see Figure~\ref{fig1}.
Taking $\varepsilon>0$ smaller if necessary, we get $\sigma_{\theta^*} (\widetilde{B}_\varepsilon) \subset \Omega$ and $\mathbf{x} \in \sigma_{\theta^*} (\partial {\widetilde{B}_\varepsilon})$. 
Since for any $\mathbf{y} \in \mathbb{R}^N$ the mapping $\theta \mapsto \sigma_{\theta}(\mathbf{y})$ is continuous, the definition of $\theta^*$ yields $u<0$ in $\sigma_{\theta^*} ({\widetilde{B}_\varepsilon})$. 
Since $u \in C^1(\Omega)$, we apply the boundary point lemma (see, e.g., \cite[Theorem~5]{vazquez}) in $\sigma_{\theta^*} ({\widetilde{B}_\varepsilon})$ and deduce that $|\nabla u(\mathbf{x})| > 0$.
This implies that the $p$-Laplacian is strictly elliptic in a neighborhood of $\mathbf{x}$, and hence $u$ is $C^2$-smooth in this neighborhood. 
Using the implicit function theorem, we conclude that $\mathcal{Z}(u)$ is a graph of a $C^{2}$-function in a neighborhood of $\mathbf{x}$, and the interior ball condition in $\Omega^\pm_u$ at $\mathbf{x}$ follows.
\end{proof}

\begin{lemma}\label{lemma_contradiction}
	Let $(O_1)-(O_4)$ be satisfied. Then the function $\alpha P_{\theta^*} u^+ + \beta P_{\theta^*} u^-$ given by Lemma~\ref{lemma_moving_polarization} is not an eigenfunction of \eqref{eq:D}.
\end{lemma}
\begin{proof}
Denote $w = P_{\theta^*} u$, for brevity. 
Take any $\mathbf{x} \in \partial \mathcal{Y}(u)$. In particular, we have $\sigma_{\theta^*}(\mathbf{x}) \in \partial \Omega$. 
Due to Lemma~\ref{lemma_contact_set_properties}, for any sufficiently small $\varepsilon>0$, 
the ball $B_\varepsilon(\mathbf{x})$ intersects $\Omega^+_u$, while $B_\varepsilon(\sigma_{\theta^*}(\mathbf{x}))$ does not (since we assume that $u<0$ in a neighborhood of $\partial \Omega$). 
Hence, considering any $\mathbf{y} \in B_\varepsilon(\mathbf{x}) \cap \Omega^+_u$, we see that $w(\mathbf{y}) = u(\sigma_{\theta^*}(\mathbf{y})) < 0$ and $w(\sigma_{\theta^*}(\mathbf{y})) = u(\mathbf{y}) > 0$. 
Moreover, by Lemma~\ref{lemma_contact_set_properties}, there exists a sufficiently small ball $\widetilde{B}_\delta \subset B_\varepsilon(\mathbf{x}) \cap \Omega^+_u$ such that $\mathbf{x} \in \partial \widetilde{B}_\delta$. Therefore, $w>0$ in $\sigma_{\theta^*}(\widetilde{B}_\delta)$ and $\Omega$ satisfies the interior ball condition at $\sigma_{\theta^*}(\mathbf{x})$. 
Due to $(O_3)$ and the discussion in Section~\ref{sec:intro}, we get $w \in C^{1}(\Omega \cup (\partial \Omega \cap B_\varepsilon(\sigma_{\theta^*}(\mathbf{x}))))$, and consequently the boundary point lemma (see, e.g., \cite[Theorem~5]{vazquez}) yields
\begin{equation}\label{eq:hopf>}
	\frac {\partial w}{\partial \nu}(\sigma_{\theta^*}(\mathbf{x})) < 0,
\end{equation}
where $\nu$ is the unit outward normal vector to $\partial \Omega$. 

On the other hand, since $\Omega$ is of class $C^{1,\delta}$ in a neighborhood of $\sigma_{\theta^*}(\mathbf{x})$ by $(O_3)$, $\sigma_{\theta^*}(\mathcal{Z}(u))$ is of class $C^2$ in a neighborhood of $\sigma_{\theta^*}(\mathbf{x})$ by Lemma~\ref{lemma_contact_set_properties}, and $\mathbf{x}$ belongs to $\partial \mathcal{Y}(u)$, we obtain the existence of a sequence 
\begin{equation}\label{eq:nonempty_intersection_1}
\{\mathbf{y}_k\} \subset \partial \Omega \setminus \sigma_{\theta^*}(\mathcal{Z}(u))
\quad \text{such that}~ \mathbf{y}_k \to \sigma_{\theta^*}(\mathbf{x})
\quad \text{as}~ k \to +\infty.
\end{equation}
By the definition of $\theta^*$, we have $u(\sigma_{\theta^*}(\mathbf{y}_k)) < 0$ for any $k$. 
Recalling also that $\{\mathbf{y}_k\} \subset \overline{\Omega^-_u}$ by the assumption, we conclude that $\{\mathbf{y}_k\} \subset \overline{\Omega^-_w}$. 
Therefore, since $w$ is $C^1$-regular up to $\partial \Omega \cap B_\varepsilon(\sigma_{\theta^*}(\mathbf{x}))$, we obtain
$$
\frac {\partial w}{\partial \nu}(\mathbf{y}_k) \geqslant 0 \quad \text{for any sufficiently large}~ k,
$$ 
which contradicts \eqref{eq:hopf>}.
\end{proof}

Combining Lemmas~\ref{lemma_moving_polarization} and \ref{lemma_contradiction}, we arrive at a contradiction, which implies that \eqref{eq:dist1} does not hold, and hence Theorem \ref{thm:main} is valid. 
\qed


\addcontentsline{toc}{section}{\refname}
	\small
	\begin{thebibliography}{99}
	
		\bibitem{allesandrini}
		Alessandrini, G. (1994). Nodal lines of eigenfunctions of the fixed membrane problem in general convex domains. Commentarii Mathematici Helvetici, 69(1), 142-154.
		\doi{10.1007/BF02564478}
		
		\bibitem{anane1987}
		Anane, A. (1987).
		Simplicit\'e et isolation de la premiere valeur propre du $p$-laplacien avec poids.
		Comptes Rendus de l'Acad\'emie des Sciences-Series I-Mathematics,
		305(16), 725-728.
		\url{http://gallica.bnf.fr/ark:/12148/bpt6k57447681/f27}
			
		\bibitem{baernstain}
		Baernstein II, A., Drasin, D., \& Laugesen, R. (2019). Symmetrization in analysis. Cambridge University Press.	
		\doi{10.1017/9781139020244}
			
		\bibitem{BartschWethWillem} 
		Bartsch, T., Weth, T., \& Willem, M. (2005). Partial symmetry of least energy nodal solutions to some variational problems. Journal d'Analyse Math\'ematique, 96(1), 1-18.
		\doi{10.1007/bf02787822}

		\bibitem{Bob1}
		Bobkov, V. (2014). Least energy nodal solutions for elliptic equations with indefinite nonlinearity. Electronic Journal of Qualitative Theory of Differential Equations, (56), 1-15.
		\doi{10.14232/ejqtde.2014.1.56}
		
		\bibitem{BK1}
		Bobkov, V., \& Kolonitskii, S. (2019). On a property of the nodal set of least energy sign-changing solutions for quasilinear elliptic equations. Proceedings of the Royal Society of Edinburgh Section A: Mathematics, 149(5), 1163-1173.
		\doi{10.1017/prm.2018.88}
		
		\bibitem{BK2x}
		Bobkov, V., \& Kolonitskii, S. B. (2024). Payne nodal set conjecture for the fractional $p$-Laplacian in Steiner symmetric domains. Zapiski Nauchnykh Seminarov POMI, 536(0), 96-125.
		\url{https://www.mathnet.ru/eng/znsl7506}
		
		\bibitem{brasco-franzina1}
		Brasco, L., \& Franzina, G. (2013). On the Hong-Krahn-Szego inequality for the $p$-Laplace operator. Manuscripta mathematica, 141(3), 537-557.
		\doi{10.1007/s00229-012-0582-x}
		
		\bibitem{brocksol}	
		Brock, F., \& Solynin, A. (2000). An approach to symmetrization via polarization. Transactions of the American Mathematical Society, 352(4), 1759-1796.
		\doi{10.1090/s0002-9947-99-02558-1}
	
		\bibitem{CuestaFucik}
		Cuesta, M., De Figueiredo, D., \& Gossez, J. P. (1999). The beginning of the Fu\v{c}ik spectrum for the $p$-Laplacian. Journal of Differential Equations, 159(1), 212-238.
		\doi{10.1006/jdeq.1999.3645}
		
		\bibitem{DGSH}
		Dahne, J., G\'omez-Serrano, J., \& Hou, K. (2021). A counterexample to Payne's nodal line conjecture with few holes. Communications in Nonlinear Science and Numerical Simulation, 103, 105957.
		\doi{10.1016/j.cnsns.2021.105957}
		
		\bibitem{damascelliNodal}
		Damascelli, L. (2000). On the nodal set of the second eigenfunction of the Laplacian in symmetric domains in $\mathbb{R}^{N}$. Atti della Accademia Nazionale dei Lincei. Classe di Scienze Fisiche, Matematiche e Naturali. Rendiconti Lincei. Matematica e Applicazioni, 11(3), 175-181.
		\url{http://eudml.org/doc/252373}		
		
		\bibitem{Dib}
		DiBenedetto, E. (1983). $C^{1+\alpha}$ local regularity of weak solutions of degenerate elliptic equations. Nonlinear Analysis: Theory, Methods \& Applications, 7(8), 827-850.
		\doi{10.1016/0362-546X(83)90061-5}
		
		\bibitem{drabekrobinson}
		Dr\'abek, P., \& Robinson, S. B. (1999). Resonance problems for the $p$-Laplacian. Journal of Functional Analysis, 169(1), 189-200.
		\doi{10.1006/jfan.1999.3501}

		\bibitem{dubinin}
		Dubinin, V. N., (2014). Condenser capacities and symmetrization in geometric function theory. Springer Basel.
		\doi{10.1007/978-3-0348-0843-9}
		
		\bibitem{fournais}
		Fournais, S. (2001). The nodal surface of the second eigenfunction of the Laplacian in $\mathbb{R}^D$ can be closed. Journal of Differential Equations, 173(1), 145-159.
		\doi{10.1006/jdeq.2000.3868}
					
		\bibitem{freitas-krej1}
		Freitas, P., \& Krej\v{c}i\v{r}\'ik, D. (2007). Unbounded planar domains whose second nodal line does not touch the boundary. Mathematical Research Letters, 14(1), 107-111.
		\doi{10.4310/MRL.2007.v14.n1.a9}
		
		\bibitem{FK}
		Freitas, P., \& Krej\v{c}i\v{r}\'ik, D. (2008). Location of the nodal set for thin curved tubes. Indiana University Mathematics Journal, 57(1), 343-375.
		\doi{10.1512/iumj.2008.57.3170}
		
		\bibitem{freitas-leylekian}
		Freitas, P., \& Leylekian, R. (2025). Payne's nodal line conjecture fails on doubly-connected planar domains. \href{https://arxiv.org/abs/2510.24436}{arXiv:2510.24436}
		
		\bibitem{jerison1}
		Jerison, D. (1995). The first nodal set of a convex domain. In Essays on Fourier analysis in honor of Elias M. Stein (Vol. 42, pp. 225-249). Princeton University Press Princeton, NJ.
		\doi{10.1515/9781400852949.225}
	
		\bibitem{HoffmannOstenhof}
		Hoffmann-Ostenhof, M., Hoffmann-Ostenhof, T., \& Nadirashvili, N. (1997). The nodal line of the second eigenfunction of the Laplacian in $\mathbb{R}^2$ can be closed. Duke Mathematical Journal, 90(3), 631-640.
		\doi{10.1215/S0012-7094-97-09017-7}
	
		\bibitem{kawohl-lind}
		Kawohl, B., \& Lindqvist, P. (2006). Positive eigenfunctions for the $p$-Laplace operator revisited. Analysis - International Mathematical Journal of Analysis and its Application, 26(4), 545.
		\doi{10.1524/anly.2006.26.4.545}
		
		\bibitem{kennedy}
		Kennedy, J. B. (2013). Closed nodal surfaces for simply connected domains in higher dimensions. Indiana University Mathematics Journal, 62(3), 785-798.
		\doi{10.1512/iumj.2013.62.4975}
		
		\bibitem{kiwan}
		Kiwan, R. (2018). On the nodal set of a second Dirichlet eigenfunction in a doubly connected domain. Annales de la Facult\'e des sciences de Toulouse: Math\'ematiques, 27(4), 863-873.
		\doi{10.5802/afst.1585}

		\bibitem{krej-tusek}
		Krej\v{c}i\v{r}\'ik, D., \& Tu\v{s}ek, M. (2015). 
		Nodal sets of thin curved layers. Journal of Differential Equations, 258(2), 281-301.
		\doi{10.1016/j.jde.2014.09.009}
		
		\bibitem{lieberman}
		Lieberman, G. M. (1988). Boundary regularity for solutions of degenerate elliptic equations. Nonlinear Analysis: Theory, Methods \& Applications, 12(11), 1203-1219. 
		\doi{10.1016/0362-546x(88)90053-3}
		
		\bibitem{melas}
		Melas, A. D. (1992). On the nodal line of the second eigenfunction of the Laplacian in $\mathbb{R}^2$. Journal of Differential Geometry, 35(1), 255-263.
		\doi{10.4310/jdg/1214447811}
				
		\bibitem{mukh-saha}
		Mukherjee, M., \& Saha, S. (2025). On the effects of small perturbation on low energy Laplace eigenfunctions. Journal of Spectral Theory, 15(3), 1045-1087.
		\doi{10.4171/JST/570}
		
		\bibitem{payne1967}
		Payne, L. E. (1967). Isoperimetric inequalities and their applications. SIAM Review, 9(3), 453-488.
		\doi{10.1137/1009070}
		
		\bibitem{payne1973}
		Payne, L. E. (1973). On two conjectures in the fixed membrane eigenvalue problem. Zeitschrift f\"ur angewandte Mathematik und Physik, 24(5), 721-729.
		\doi{10.1007/BF01597076}

		\bibitem{Polya1}
		P\'olya, G. (1950). Sur la sym\'etrisation circulaire. Comptes Rendus Hebdomadaires des Seances de L Academie Des Sciences, 230(1), 25-27.
		
		
		\bibitem{putter}
		P\"utter, R. (1990). On the nodal lines of second eigenfunctions of the fixed membrane problem. Commentarii Mathematici Helvetici, 65(1), 96-103.
		\doi{10.1007/BF02566596}
		
		\bibitem{sarswat}
		Sarswat, A. (2014). On the nodal line of a second eigenfunction of the Laplacian-Dirichlet in some annular domains with dihedral symmetry. 
		\href{https://arxiv.org/abs/1411.0221}{arXiv:1411.0221}.
		
		\bibitem{sarvas}
		Sarvas, J. (1973). Symmetrization of condensers in $n$-space. Annales Fennici Mathematici, (522).
		\doi{10.5186/aasfm.1973.522}
		
		\bibitem{YG}
		Yang, D. H., \& Guo, B. Z. (2013). On nodal line of the second eigenfunction of the Laplacian over concave domains in $\mathbb{R}^2$. Journal of Systems Science and Complexity, 26(3), 483-488.
		\doi{10.1007/s11424-013-1175-9}

		\bibitem{vazquez}
		V\'azquez, J. L. (1984). A strong maximum principle for some quasilinear elliptic equations. Applied Mathematics and Optimization, 12(1), 191-202.
		\doi{10.1007/bf01449041}
		
	\end{thebibliography}
\end{document}